\documentclass[12pt]{amsart}
\usepackage{amsthm}
\usepackage{amssymb}
\usepackage{listings}
\usepackage{longtable}
\usepackage{geometry}
\usepackage{enumerate}
\usepackage{setspace}
\usepackage{pdflscape}
\usepackage{upgreek}
\usepackage{tikz}
\usepackage{tikz-cd}
\usetikzlibrary{matrix, arrows}
\usepackage[unicode=true]
 {hyperref}
\hypersetup{
colorlinks=true,
urlcolor=black,
citecolor=blue,
linkcolor=blue,
}

\allowdisplaybreaks

\newtheorem{thm}{Theorem}[section]
\newtheorem{prop}[thm]{Proposition}
\newtheorem{lem}[thm]{Lemma}

\newtheorem{conjecture}[thm]{Conjecture}

\theoremstyle{definition}

\newtheorem{question}[thm]{Question}

\theoremstyle{remark}
\newtheorem{remark}[thm]{Remark}

\numberwithin{equation}{section}

\makeatletter
\newcommand{\addresseshere}{%
  \enddoc@text\let\enddoc@text\relax
}
\makeatother

\newcommand{\ep}{\epsilon}
\newcommand{\Perm}{\mathrm{Perm}}
\newcommand{\Hol}{\mathrm{Hol}}
\newcommand{\Aut}{\mathrm{Aut}}

\newcommand{\Hom}{\mathrm{Hom}}
\newcommand{\Map}{\mathrm{Map}}
\newcommand{\Inn}{\mathrm{Inn}}
\newcommand{\Out}{\mathrm{Out}}
\newcommand{\proj}{\mathrm{proj}}
\newcommand{\ff}{\mathfrak{f}}
\newcommand{\fg}{\mathfrak{g}}
\newcommand{\bZ}{\mathbb{Z}}
\newcommand{\bN}{\mathbb{N}}

\newcommand{\PSL}{\mathrm{PSL}}
\newcommand{\SL}{\mathrm{SL}}
\newcommand{\GL}{\mathrm{GL}}
\newcommand{\E}{\mathcal{E}}
\newcommand{\Fit}{\mathrm{Fit}}
\newcommand{\Sz}{\mathrm{Sz}}
\newcommand{\res}{\mathrm{res}}

\allowdisplaybreaks

\begin{document}

\large 

\title[Solvability of regular subgroups in the holomorph]{On the solvability of regular subgroups \\in the holomorph of a finite solvable group}

\author{Cindy (Sin Yi) Tsang}
\address{School of Mathematics (Zhuhai), Sun Yat-Sen University, P. R. China}
\email{zengshy26@mail.sysu.edu.cn}
\urladdr{http://sites.google.com/site/cindysinyitsang/} 

\author{Chao Qin}
\address{School of Mathematics (Zhuhai), Sun Yat-Sen University, P. R. China}
\email{qinch23@mail.sysu.edu.cn}

\date{\today}

\maketitle

\begin{abstract}We exhibit infinitely many natural numbers $n$ for which there exists at least one insolvable group of order $n$, and yet the holomorph of any solvable group of order $n$ has no insolvable regular subgroup. We also solve Problem 19.90 (d) in the Kourovka notebook.
\end{abstract}

\tableofcontents 


\section{Introduction}

Let $N$ be a finite group and write $\Perm(N)$ for its symmetric group. First recall that a subgroup $\mathcal{G}$ of $\Perm(N)$ is said to be \emph{regular} if the map
\[ \xi_\mathcal{G}:\mathcal{G}\longrightarrow N;\hspace{1em}\xi_\mathcal{G}(\sigma) = \sigma(1_N)\]
is bijective, or equivalently, if the $\mathcal{G}$-action on $N$ is both transitive and free. For example, the images of the left and right regular representations
\[\begin{cases}
\lambda: N\longrightarrow \Perm(N);\hspace{1em}\lambda(\eta) = (x\mapsto \eta x),\\
\rho:N\longrightarrow\Perm(N);\hspace{1em}\rho(\eta) = (x\mapsto x\eta^{-1}),
\end{cases}\]
respectively, are both regular subgroups of $\Perm(N)$. Plainly, a regular subgroup of $\Perm(N)$ has the same order as $N$, but is not necessarily isomorphic to $N$ in general. Given a group $G$ of order $|N|$, define
\[\E'(G,N) = \left\{\mbox{regular subgroups of $\Hol(N)$ isomorphic to $G$}\right\},\]
where $\Hol(N)$ denotes the \emph{holomorph of $N$} and is given by
\begin{equation}\label{Hol(N)}
\Hol(N) = \rho(N)\rtimes \Aut(N).
\end{equation}
This set $\E'(G,N)$ is an important object in the studies of Hopf-Galois structures and skew braces; see \cite[Chapter 2]{Childs book} and \cite{Skew}, respectively. In particular, there is a connection between elements of $\E'(G,N)$ and
\begin{enumerate}[1.]
\item Hopf-Galois structures of type $N$ on a Galois extension with group $G$;
\item skew braces with additive group $N$ and multiplicative group $G$.
\end{enumerate}
Let us remark that skew braces in turn are closely related to non-degenerate set-theoretic solutions to the Yang-Baxter equation; see \cite{YBE}.

\vspace{1.5mm}

Observe that $\E'(G,N)$ contains $\lambda(N)$ and $\rho(N)$ when $G\simeq N$. However, in general $\E'(G,N)$ might be empty when $G\not\simeq N$. It is natural to ask:

\begin{question}For the set $\E'(G,N)$ to be non-empty, what are some restrictions on $G$ and $N$ in terms of their group-theoretic properties?
\end{question}

This question was studied by N. P. Byott in \cite{Byott soluble}, where he showed that:

\begin{prop}\label{prop1}Let $G$ and $N$ be  two finite groups of the same order such that the set $\E'(G,N)$ is non-empty.
\begin{enumerate}[(a)]
\item If $N$ is nilpotent, then $G$ is solvable.
\item If $G$ is abelian, then $N$ is solvable.
\end{enumerate}
\end{prop}
\begin{proof}See \cite[Theorems 1 and 2]{Byott soluble}. 
\end{proof}

In fact, the proof of Proposition~\ref{prop1} (b) from \cite[Section 6]{Byott soluble} may be used to show the following stronger result. This was observed by the first author in \cite[Theorem 4.2.4]{Tsang}, which is unpublished, and we shall reproduce the proof in Section~\ref{proof sec1} below. Let us remark that Theorem~\ref{thm1} (c) solves Problem 19.90 (d) in the Kourovka notebook \cite{K book}.

\begin{thm}\label{thm1}Let $G$ and $N$ be two finite groups of the same order such that the set $\E'(G,N)$ is non-empty.
\begin{enumerate}[(a)]
\item If $G$ is cyclic, then $N$ is supersolvable.
\item If $G$ is abelian, then $N$ is metabelian.
\item If $G$ is nilpotent, then $N$ is solvable.
\end{enumerate}
\end{thm}

In the proof of \cite[Corollary 1.1]{Byott soluble}, N. P. Byott gave examples of solvable $G$ and insolvable $N$ with non-empty $\E'(G,N)$. Also, he noted that by contrast, so far there is no known example of
\begin{equation}\label{GN conditions}\mbox{insolvable $G$ and solvable $N$ with non-empty $\E'(G,N)$}.\end{equation}
Results in the literature suggest that in fact no such example exists. 

\begin{prop}\label{prop2} Let $G$ and $N$ be two finite groups of the same order such that the set $\E'(G,N)$ is non-empty.
\begin{enumerate}[(a)]
\item If $G$ is non-abelian simple, then $N\simeq G$.
\item If $G$ is the double cover of $A_m$ with $m\geq5$, then $N\simeq G$.
\item If $G$ is $S_m$ with $m\geq 5$, then $N$ contains an isomorphic copy of $A_m$.
\end{enumerate}
Here $A_m$ and $S_m$ denote, respectively, the alternating and symmetric groups on $m$ letters.
\end{prop}
\begin{proof}
See \cite[Theorem 1.1]{Byott simple}, \cite[Theorem 1.6]{Tsang HG}, and \cite[Theorem 1.3]{Tsang Sn}.
\end{proof}

It leads us to the following conjecture. It was N. P. Byott who told the first author about this problem in person and Conjecture~\ref{conjecture} should be attributed to him.

\begin{conjecture}\label{conjecture}For any $n\in\bN$, there do not exist finite groups $G$ and $N$ both of order $n$ for which (\ref{GN conditions}) holds.
\end{conjecture}

In Section~\ref{criteria sec}, using techniques developed by the first author in \cite[Section 4.1]{Tsang HG}, we shall prove some necessary criteria for $\E'(G,N)$ to be non-empty. In Sections~\ref{cubefree sec} and~\ref{sqfree sec}, by applying our criteria, we shall show that:

\begin{thm}\label{main thm1}Conjecture~$\ref{conjecture}$ holds when $n$ is cube-free.
\end{thm}

\begin{thm}\label{main thm2}Conjecture~$\ref{conjecture}$ holds when $ n = 2^r\cdot n_0$ with
\[ n_0 = 2^2\cdot 3\cdot 5,\,2^4\cdot 3^2\cdot 17,\mbox{ or }4^{\ell_0}(4^{\ell_0}+1)(2^{\ell_0}-1),\]
where $\ell_0$ is any odd prime such that $(4^{\ell_0}+1)(2^{\ell_0}-1)$ is square-free and $r$ is any non-negative integer.\end{thm}

\begin{remark}\label{remark1}The numbers $n_0$ in Theorem~\ref{main thm2} are significant because
\begin{align}\notag
|A_5| & = 2^2\cdot 3\cdot 5,\\\notag
|\PSL_2(17)| &= 2^4\cdot 3^2\cdot 17,\\\label{Sz order}
|\Sz(2^{2m+1})| & = 4^{2m+1}(4^{2m+1}+1)(2^{2m+1}-1)\mbox{ for $m\in\bN$},
\end{align}
where $\Sz(-)$ denotes the Suzuki groups \cite{Sz}, and there is a unique insolvable group of order $n_0$ which is non-abelian simple; see Lemmas~\ref{lem1'} and~\ref{lem2'}. Also, the key is that they satisfy the special conditions in Theorem~$\ref{general thm}$ below.
\end{remark}

\begin{remark}\label{remark2}Let $\ell_0$ be an odd prime and let us discuss how often
\[(4^{\ell_0}+1)(2^{\ell_0}-1)\]
is square-free. Note that $4^5+1$ is divisible by $25$, so let us assume that $\ell_0\neq5$. 

\vspace{1.5mm}

Suppose that $p$ is a prime and $p^2$ divides $(4^{\ell_0}+1)(2^{\ell_0}-1)$. Clearly $p\geq5$ and $p$ cannot divide both $4^{\ell_0}+1$ and $2^{\ell_0}-1$. We shall show that $p$ must be a Wieferich prime, namely $2^{p-1}\equiv1$ (mod $p^2$). We thank one of the referees for pointing out this relation with Wieferich primes. If
\[ 2^{\ell_0}\equiv1\mbox{ (mod $p^2$)},\]
then $\ell_0\mid p-1$ and $p$ is clearly a Wieferich prime. If
\[ 4^{\ell_0}\equiv-1\mbox{ (mod $p^2$)},\mbox{ and in particular }2^{4\ell_0}\equiv16^{\ell_0}\equiv1\mbox{ (mod $p^2$)},\]
then $-1$ is a square mod $p$ and $4\mid p-1$. Since $4^{10}\equiv1$ (mod $25$) and $\ell_0\neq5$, it also implies that $p\neq5$. Thus, we have $p\geq 7$ and so $16\not\equiv1$ (mod $p$). Then, it follows that $\ell_0\mid p-1$, whence $4\ell_0\mid p-1$ and we see that $p$ is a Wieferich prime.

\vspace{1.5mm}

Except $1093$ and $3511$, there is no Wieferich prime less than $4\times 10^{12}$ by \cite{Wprimes}. This suggests that $(4^{\ell_0}+1)(2^{\ell_0}-1)$ is square-free for most $\ell_0\geq7$, if not all.
\end{remark}

In Section~\ref{sec algorithm}, we shall also present an algorithm which may be used to show that Conjecture~\ref{conjecture} holds for any given $n\in\bN$, given that all finite groups of order $n$ have been classified. By implementing our algorithm in {\sc Magma} \cite{magma} and using the {\sc SmallGroups} Library \cite{SGlibrary}, we verified that:

\begin{thm}\label{main thm3}Conjecture~$\ref{conjecture}$ holds when $n\leq 2000$.
\end{thm}

A natural number $n$ is called \emph{solvable} if every group of order $n$ is solvable, and is called \emph{non-solvable} otherwise. Conjecture~\ref{conjecture} is of course trivial when $n$ is a solvable number. Since any multiple of a non-solvable number is again non-solvable, the numbers $n$ in Theorem~\ref{main thm2} are non-solvable by Remark~\ref{remark1}. See \cite[A056866]{OEIS} for a complete list of non-solvable numbers at most $2000$.

\section{Proof of Theorem~\ref{thm1}}\label{proof sec1}


Let $N$ be a finite group and let $\mathcal{G}$ be any regular subgroup of $\Hol(N)$. Let
\[ \proj_\rho : \Hol(N)\longrightarrow \rho(N) \mbox{ and }\proj_\Aut: \Hol(N)\longrightarrow \Aut(N),\]
respectively, denote the projection map and homomorphism afforded by (\ref{Hol(N)}). Since $\mathcal{G}$ is regular, we easily verify that $(\proj_\rho)|_{\mathcal{G}}$ is bijective and that
\[\rho(N)\rtimes \proj_\Aut(\mathcal{G}) = \mathcal{G}\cdot\proj_\Aut(\mathcal{G}).\]
Theorem~\ref{thm1} then follows directly from Lemmas~\ref{lem product} and~\ref{lem properties} below.

\begin{lem}\label{lem product}Let $\Gamma$ be a finite group which is a product of two subgroups $\Delta_1$ and $\Delta_2$, namely, elements of $\Gamma$ are of the shape $\delta_1\delta_2$ with $\delta_1\in\Delta_1,\delta_2\in\Delta_2$.
\begin{enumerate}[(a)]
\item If $\Delta_1$ and $\Delta_2$ are cyclic, then $\Gamma$ is supersolvable.
\item If $\Delta_1$ and $\Delta_2$ are abelian, then $\Gamma$ is metabelian.
\item If $\Delta_1$ and $\Delta_2$ are nilpotent, then $\Gamma$ is solvable.
\end{enumerate}
\end{lem}
\begin{proof}This is known, by \cite{Douglas}, \cite{Ito}, and \cite{Kegel}, respectively. \end{proof}

\begin{lem}\label{lem properties}The properties ``cyclic'', ``abelian'', ``nilpotent'', ``supersolvable'', ``metabelian'', ``solvable" are all quotient-closed and subgroup-closed.
\end{lem}
\begin{proof}For ``cyclic'' and ``abelian'', this is obvious. For ``nilpotent'' and ``supersolvable'', a proof may be found in \cite[Theorems 10.3.1 and 10.5.1]{Hall}. As for ``metabelian'' and ``solvable'', see \cite[3.10 and the discussion after 3.11]{Isaacs}.\end{proof}

\section{Criteria for non-emptiness}\label{criteria sec}

Throughout this section, assume that $G$ and $N$ are two finite groups of the same order such that the set $\E'(G,N)$ is non-empty. Then, as noticed in \cite[Proposition 2.1]{Tsang HG}, for example, by (\ref{Hol(N)}) this is equivalent to the existence of
\[ \ff\in\Hom(G,\Aut(N))\mbox{ and bijective }\fg\in\Map(G,N)\]
satisfying the relation
\begin{equation}\label{fg relation}
\fg(\sigma\tau) = \fg(\sigma)\cdot \ff(\sigma)(\fg(\tau)) \mbox{ for all }\sigma,\tau\in G.
\end{equation}
Below, we shall use (\ref{fg relation}) to give two necessary relations between $G$ and $N$, both of which are not very hard to prove. Yet, the criterion in Proposition~\ref{test2} seems to be fairly powerful, and it alone allows us to prove Theorems~\ref{main thm1} and~\ref{main thm2}. Also, let us recall the following useful fact.

\begin{lem}\label{solvable lem}Let $\Gamma$ be a group containing a normal subgroup $\Delta$. Then, the group $\Gamma$ is solvable if and only if both $\Delta$ and $\Gamma/\Delta$ are solvable.
\end{lem}
\begin{proof}This is a standard result; see \cite[3.10]{Isaacs}, for example.\end{proof}

To state the first criterion, let $\Inn(N)$ and $\Out(N)$, denote the inner and outer automorphism groups of $N$, respectively. Let $\uppi:\Aut(N)\longrightarrow\Out(N)$ denote the natural quotient map with kernel equal to $\Inn(N)$. Then, we have:

\begin{prop}\label{test1}If $G$ is insolvable and $N$ is solvable, then $(\uppi\circ\ff)(G)$ is an insolvable subgroup of $\Out(N)$.
\end{prop}
\begin{proof}Observe that $\ff$ induces an embedding
\[ \ker(\uppi\circ\ff)/\ker(\ff) \longrightarrow \Inn(N)\]
and that $\fg$ restricts to a homomorphism $\ker(\ff)\longrightarrow N$ by (\ref{fg relation}). Hence, if $N$ is solvable, then both $\ker(\ff)$ and $\Inn(N)$ are solvable by Lemma~\ref{lem properties}, and so $\ker(\uppi\circ\ff)$ is solvable by Lemma~\ref{solvable lem}. If $G$ is insolvable in addition, then since
\[ G/\ker(\uppi\circ\ff) \simeq (\uppi\circ\ff)(G),\]
we see that $(\uppi\circ\ff)(G)$ is insolvable, again by Lemma~\ref{solvable lem}.
\end{proof}

To state the second criterion, let us recall that a subgroup $M$ of $N$ is called \emph{characteristic} if $\varphi(M)=M$ for all $\varphi\in\Aut(N)$. In this case, plainly $M$ is normal in $N$, and we shall write
\[\theta_M:\Aut(N)\longrightarrow\Aut(N/M);\hspace{1em}\theta_M(\varphi) = (\eta M \mapsto \varphi(\eta)M)\]
for the natural homomorphism. The use of characteristic subgroups of $N$ is motivated by the arguments in \cite{Byott simple}; also see \cite[Section 4.1]{Tsang HG}. Our main tool is the following proposition; also see Proposition~\ref{prop GN}.

\begin{prop}\label{test2}Let $M$ be any characteristic subgroup of $N$ and define
\[ H = \fg^{-1}(M).\]
Then, this set $H$ is a subgroup of $G$, and $\E'(H,M)$ is non-empty. Moreover, if $N/M$ is solvable and $\ker(\theta_M\circ\ff)$ is insolvable, then $H$ is insolvable.
\end{prop}
\begin{proof}The set $H$ is a subgroup of $G$ by (\ref{fg relation}); see \cite[Lemma 4.1]{Tsang HG}. Also, we have a homomorphism
\[ \res(\ff) : H \longrightarrow \Aut(M);\hspace{1em}\res(\ff)(\sigma) = \ff(\sigma)|_M\]
induced by $\ff$ since $M$ is characteristic, and also a bijective map
\[ \res(\fg): H \longrightarrow M;\hspace{1em}\res(\fg)(\sigma) = \fg(\sigma)\]
induced by $\fg$ since $\fg$ is bijective. Clearly, it follows directly from (\ref{fg relation}) that
\[ \res(\fg)(\sigma\tau) = \res(\fg)(\sigma)\cdot (\res(\ff)(\sigma))(\res(\fg)(\tau))\mbox{ for all }\sigma,\tau\in H.\]
Then, by \cite[Proposition 2.1]{Tsang HG}, which is a consequence of (\ref{Hol(N)}), this implies that $\E'(H,M)$ is non-empty. This proves the first statement. 

\vspace{1.5mm}

Next, as noted in \cite[Lemma 4.1]{Tsang HG}, the relation (\ref{fg relation}) implies that
\[ \ker(\theta_M\circ\ff) \longrightarrow N/M;\hspace{1em}\sigma\mapsto \fg(\sigma)M\]
induced by $\fg$ is a homomorphism, and so we have an embedding
\[ \frac{\ker(\theta_M\circ\ff)}{\ker(\theta_M\circ\ff)\cap H}\longrightarrow N/M.\]
Thus, if $N/M$ is solvable and $\ker(\theta_M\circ\ff)$ is insolvable, then $ \ker(\theta_M\circ\ff)\cap H$ must be insolvable by Lemma~\ref{solvable lem}, which in turn implies that $H$ is insolvable by Lemma~\ref{lem properties}. The second statement then follows.
\end{proof}

Although Proposition~\ref{test2} is valid for any characteristic subgroup $M$ of $N$, motivated by \cite{Byott simple}, we shall consider the case when $M$ is a (proper) maximal characteristic subgroup of $N$. In this case, the quotient $N/M$ is a non-trivial characteristically simple group, and so we know that 
\[N/M\simeq T^m,\mbox{ where $T$ is a simple group and $m\in\bN$}.\]
Hence, if $N$ is solvable, then
\begin{equation}\label{N/M}N/M\simeq(\bZ/p\bZ)^m\mbox{ and in particular }\Aut(N/M)\simeq\GL_m(p),
\end{equation}
where $p$ is a prime. The following is well-known.

\begin{lem}\label{GL lem}For any prime $p$ and $m\in\bN$, the group $\GL_m(p)$ is solvable if and only if $m=1$ or $m=2$ with $p\leq3$.
\end{lem}

\section{Proof of Theorem~\ref{main thm1}}\label{cubefree sec}

Suppose for contradiction that the claim is false and let $n$ be the smallest cube-free number for which Conjecture~\ref{conjecture} fails. Let $G$ and $N$ be two groups of order $n$ satisfying (\ref{GN conditions}). Let $M$ be any proper and maximal characteristic subgroup of $N$. Clearly $M$ is solvable because $N$ is solvable. As in (\ref{N/M}), we then know that
\[ N/M \simeq (\bZ/p\bZ)^m,\mbox{ where $p$ is a prime and $m\in\bN$}.\]
Notice that $|M| = n/p^m$ and that $m=1,2$ because $n$ is cube-free. Hence, by Lemma~\ref{insol kernel} (b) below, the kernel of any homomorphism $G\longrightarrow\Aut(N/M)$ is insolvable. From Proposition~\ref{test2}, it follows that $\E'(H,M)$ is non-empty for some insolvable subgroup $H$ of $G$ of the same order as $M$. This contradicts the minimality of $n$ and so Theorem~\ref{main thm1} must be true.

\begin{lem}\label{insol kernel}Let $p$ be any prime and let $m=1,2$.
\begin{enumerate}[(a)]
\item The group $\GL_m(p)$ has no non-abelian simple subgroup.
\item The kernel of a homomorphism from a finite insolvable group of cube-free order to $\GL_m(p)$ is insolvable.
\end{enumerate}
\end{lem}
\begin{proof}For $m=1$ or $p=2$, the group $\GL_m(p)$ is solvable by Lemma~\ref{GL lem}, and the claims hold by Lemmas~\ref{lem properties} and~\ref{solvable lem}. For $m=2$ and $p$ odd, first suppose for contradiction that $\GL_2(p)$ has a subgroup $A$ which is non-abelian simple. Observe that the homomorphism
\[ \begin{tikzcd}[column sep = 2cm]
A\arrow{r}{\tiny\mbox{inclusion}} & \GL_2(p) \arrow{r}{\tiny\mbox{determinant}} & (\bZ/p\bZ)^\times
\end{tikzcd}\]
must be trivial, and so $A$ is in fact a subgroup of $\SL_2(p)$. Also, note that $A$ has an element of order two by the Feit-Thompson theorem. Since $p$ is odd, the matrix  $\left(\begin{smallmatrix}-1&0\\ 0 & -1\end{smallmatrix}\right)$, which lies in the center, is the only element in $\SL_2(p)$ of order two. It follows that $A$ has non-trivial center, which is a contradiction. Alternatively, the subgroups of $\SL_2(p)$ have been classified; see \cite[Theorem 6.17]{Suzuki book}. None of the groups listed there  are non-abelian simple, and we obtain a contradiction. We thank one of the referees for bringing Dickson's result on the subgroups of $\PSL_2(p)$ to our attention, which led us to this simpler proof which does not use the Feit-Thompson theorem. This proves part (a). Since any insolvable group of cube-free order has a non-abelian simple subgroup by \cite{cubefree}, we see that part (b) follows from part (a) and Lemma~\ref{solvable lem}.
\end{proof}

\section{Almost square-free orders}\label{sqfree sec}

In this section, we shall prove Theorem~\ref{main thm2}. First, let us prove the following more general statement.

\begin{thm}\label{general thm}Suppose that $n_0 = 2^{r_0}\cdot 3^{\ep_0}\cdot p_1\cdots p_{k_0}$, where
\[r_0,k_0\in\mathbb{N}_{\geq0},\, \ep_0\in\{0,1,2\},\mbox{ and }p_1,\dots,p_{k_0}\geq5\mbox{ are distinct primes},\]
and that Conjecture~\ref{conjecture} holds when $n = n_0$. Assume that the following hold.
\begin{enumerate}[1.]
\item the subgroups of index a power of two of an insolvable group of order $n_0$ are all insolvable;
\item there is no non-abelian simple group of order $2^r\cdot n_0$ for $r\in\bN$;
\item the number $n_0/2$ is solvable in the case that $n_0$ is even;
\item the numbers $(2^r\cdot n_0)/p$, where $p$ ranges over the odd primes dividing $n_0$, are all solvable for $r\in\bN_{\geq0}$.
\end{enumerate}
Then Conjecture~\ref{conjecture} also holds when $n = 2^r\cdot n_0$ for any $r\in\bN$.
\end{thm}
\begin{proof}Suppose for contradiction that the four conditions are satisfied but the conclusion is false. Let $r\in\bN$ be the smallest number such that Conjecture~\ref{conjecture} does not hold when $n = 2^r\cdot n_0$. Also, let $G$ and $N$ be two groups of order $n$ satisfying (\ref{GN conditions}). Let $M$ be any proper and maximal characteristic subgroup of $N$. Clearly $M$ is solvable because $N$ is solvable. As in (\ref{N/M}), we have
\[ N/M \simeq (\bZ/p\bZ)^m,\mbox{ where $p$ is a prime and $m\in\bN$}.\]
Notice that $|M| = n/ p^m$. Also, we know from Proposition~\ref{test2} that $\E'(H,M)$ is non-empty for some subgroup $H$ of $G$ of the same order as $M$. 

\vspace{1.5mm}

For $p$ odd, we have $m\leq 2$ if $p=3$ and $m=1$ if $p\geq 5$ by the hypothesis on $n_0$, so $\GL_m(p)$ is solvable by Lemma~\ref{GL lem}. Then, the kernel of any homomorphism $G\longrightarrow\Aut(N/M)$ must be insolvable by Lemma~\ref{solvable lem}, and we may take $H$ to be insolvable by Proposition~\ref{test2}, which contradicts condition 4. In the case that $\ep_0=2$, it is possible that $m=2$ when $p=3$, but note that $2^r\cdot n_0/9$ is also solvable by condition $4$ since a factor of a solvable number is solvable.

\vspace{1.5mm}

For $p=2$, we have $|H| = 2^{r-m}\cdot n_0$, and thus $H$ is insolvable by Lemma~\ref{index power of two} below. Observe that $r-m\geq 0$ by condition 3. Since Conjecture~\ref{conjecture} holds when $n=n_0$ by assumption, we in fact have $r-m\geq1$, which contradicts the minimality of $r$.\end{proof}

\begin{lem}\label{index power of two}Let $n_0\in\bN$ be any integer such that the conditions 1, 2, 3, 4 in Theorem~$\ref{general thm}$ are satisfied. Then, for any $r\in\bN_{\geq0}$, we have:
\begin{enumerate}[(i)]
\item the subgroups of index a power of two of any insolvable group of order $2^r\cdot n_0$ are insolvable;
\item any insolvable group of order $2^r\cdot n_0$ has a non-abelian composition factor of order $n_0$.
\end{enumerate}
\end{lem}
\begin{proof}Notice that since a non-solvable number is a multiple of the order of a non-abelian simple group, conditions 3 and 4 imply that an insolvable group of order $n_0$ must be non-abelian simple.

\vspace{1.5mm}

We shall use induction on $r$. For $r=0$, claim (i) is simply condition 1, and claim (ii) holds by the above observation. Suppose now that $r\geq1$, and let $G$ be an insolvable group of order $2^r\cdot n_0$. By condition 2, we know that $G$ has a non-trivial and proper normal subgroup $A$. Either $A$ or $G/A$ is insolvable by Lemma~\ref{solvable lem}. Since a factor of a solvable number is solvable, we have
\[ 2^{a}\cdot n_0 = \begin{cases}|A| & \mbox{if $A$ is insolvable},\\ |G/A| &\mbox{if $G/A$ is insolvable},\end{cases}\]
where $0\leq a\leq r-1$, conditions 3 and 4. By the induction hypothesis, either $A$ or $G/A$ has a non-abelian composition factor of order $n_0$. It follows that $G$ has a non-abelian composition factor of order $n_0$ also, which proves (ii). Next, let $H$ be a subgroup of $G$ of index a power of two. Observe that $AH/A \simeq H/A\cap H$, and also that
\begin{align*}[A:A\cap H] &= [G:H]/[G:AH],\\
[G/A :AH/A] &= [G:H]/[A:A\cap H],\end{align*}
both of which are powers of two. Hence, by the induction hypothesis, we see that either $A\cap H$ or $H/A\cap H$ is insolvable. It then follows from Lemma~\ref{lem properties} that $H$ is insolvable, which proves (i).
\end{proof}

We shall apply Theorem~$\ref{general thm}$ to prove Theorem~\ref{main thm2}. To that end, we shall first show that the numbers $n_0$ in the statement of Theorem~\ref{main thm2} satisfy conditions 1, 2, 3, 4 in Theorem~$\ref{general thm}$.

\begin{lem}\label{facts lem}The following statements are true.
\begin{enumerate}[(a)]
\item A non-solvable number is divisible by at least three distinct primes.
\item A finite non-abelian simple group whose order is not divisible by three is a Suzuki group.
\end{enumerate}
\end{lem}
\begin{proof}Part (a) is Burnside's theorem. Part (b) follows from the classification of finite simple groups.
\end{proof}

\begin{lem}\label{lem1}Let $n_0 = 2^{r_0}\cdot 3^{\ep_0}\cdot p$, where $r_0\in\bN$, $\ep_0\in\{1,2\}$, and $p\geq5$ is a prime. If there exists a non-abelian simple group $\Gamma$ of order $n_0$, then
\begin{equation}\label{n0 1}n_0 \in \{2^2\cdot 3\cdot 5,\,2^3\cdot3\cdot 7,\,2^3\cdot 3^2\cdot 7,\,2^4\cdot3^2\cdot 17,\,2^3\cdot 3^2\cdot 5 \}\end{equation}
and
\[ \Gamma\simeq \begin{cases}A_5&\mbox{for }n_0 = 2^2\cdot 3\cdot 5,\\
\PSL_2(7) &\mbox{for }n_0 = 2^3\cdot 3\cdot 7,\\
\PSL_2(8) &\mbox{for }n_0 = 2^3\cdot 3^2\cdot 7, \\
\PSL_2(17) &\mbox{for }n_0 = 2^4\cdot 3^2\cdot 17,\\
A_6 &\mbox{for }n_0 = 2^3\cdot 3^2\cdot 5.\end{cases}\]
In particular, condition $2$ in Theorem~$\ref{general thm}$ is satisfied for $n_0$ in (\ref{n0 1}).
\end{lem}
\begin{proof}Since $p$ exactly divides $n_0$, a Sylow $p$-subgroup of any group of order $n_0$ is cyclic. If $p>3^{\ep_0}$, then the claim follows from \cite[Theorem 1]{Herzog}. If not, then $\ep_0=2$ with $p = 5,7$, and the claim follows from \cite{Brauer} and \cite{Wales}, respectively. 
\end{proof}

\begin{lem}\label{lem1'}Let $n_0 = 2^2\cdot 3\cdot 5$ or $2^4\cdot 3^2\cdot 17$. Then, up to isomorphism $A_5$ or $\PSL_2(17)$, respectively, is the only insolvable group of order $n_0$. Moreover, conditions $1, 3, 4$ in Theorem~$\ref{general thm}$ are satisfied.
\end{lem}
\begin{proof}Since a non-solvable number is a multiple of the order of a non-abelian simple group, from Lemmas~\ref{facts lem} (a) and~\ref{lem1}, it is easy to deduce the first claim and that conditions 3 and 4 hold. Condition 1 holds trivially because $A_5$ and $\PSL_2(17)$ have no proper subgroup of index a power of two.
\end{proof}

Note that $n_0 = 2^3\cdot 3\cdot 7$ fails condition 1  while $n_0=2^3\cdot 3^2\cdot 7$ and $2^3\cdot 3^2\cdot 5$ fail condition 4 in Theorem~$\ref{general thm}$.

\begin{lem}\label{lem2}Let $n_0 = 2^{r_0}(4^{2m_0+1}+1)(2^{2m_0+1}-1)$, where $r_0,m_0\in\bN$. If there exists a non-abelian simple group $\Gamma$ of order $n_0$, then 
\[r_0 = 2(2m_0+1)\mbox{ with }\Gamma\simeq \Sz(2^{2m_0+1}).\]
In particular, condition $2$ in Theorem~$\ref{general thm}$ is satisfied for $r_0=2(2m_0+1)$.
\end{lem}
\begin{proof}This is clear from Lemma~\ref{facts lem} (b) and (\ref{Sz order}).
\end{proof}

\begin{lem}\label{lem2'}Let $n_0 = 4^{\ell_0}(4^{\ell_0}+1)(2^{\ell_0}-1)$, where $\ell_0$ is an odd prime. Then, up to isomorphism  $\Sz(2^{\ell_0})$ is the only insolvable group of order $n_0$. Moreover, conditions $1, 3, 4$ in Theorem~$\ref{general thm}$ are satisfied.
\end{lem}
\begin{proof}Suppose for contradiction that there is an insolvable group of order $n_0$ which is not isomorphic to $\Sz(2^{\ell_0})$, and thus cannot be non-abelian simple by Lemma~\ref{lem2}. Since a non-solvable number is the multiple of the order of a non-abelian simple group, from Lemma~\ref{facts lem} (b) and (\ref{Sz order}), we deduce that
\[ 4^{\ell_0}(4^{\ell_0}+1)(2^{\ell_0}-1) = n_0 = d\cdot  4^{k}(4^{k}+1)(2^{k}-1),\]
where $d,k\in\bN$ with $k\geq3$ odd and $d\geq2$. Plainly $\ell_0\neq k$, and because $\ell_0$ is prime, we deduce that
\[ \gcd(2^{k}-1,2^{\ell_0}-1) = 2^{\gcd(k,\ell_0)} - 1 = 2 - 1 =1.\]
This means that $2^{k}-1$ divides $4^{\ell_0}+1$. Note that then $k\leq 2\ell_0$. But 
\[ (2^{k}-1) + (2^{2\ell_0 - tk}+1) = 2^{k}(2^{2\ell_0 - (t+1)k} + 1)\mbox{ for all $t\in\bN_{\geq0}$}.\]
By induction, this implies that $2^{k}-1$ divides $2^s+1$ for some $0\leq s\leq k-1$, which is impossible because $k\geq3$. This proves the first claim. 

\vspace{1.5mm}

Now, the maximal subgroups of $\Sz(2^{\ell_0})$ are known; see \cite[Theorem 4.1]{Wilson}, for example. None has index a non-trivial power of two and so condition 1 is satisfied. To prove conditions 3 and 4, note that if $n_0/2$ were non-solvable, then it would be a multiple of the order of a non-abelian simple group, so by Lemma~\ref{facts lem} (b) and (\ref{Sz order}), we have
\[ 4^{\ell_0}(4^{\ell_0}+1)(2^{\ell_0}-1) = 2\cdot d\cdot  4^{k}(4^{k}+1)(2^{k}-1),\]
where $d,k\in\bN$ with $k\geq3$ odd. Similarly, if $(2^r\cdot n_0)/p$ were non-solvable for some odd prime $p$ divisor of $n_0$ and $r\in\bN_{\geq0}$, then we have
\[ 2^r\cdot 4^{\ell_0}(4^{\ell_0}+1)(2^{\ell_0}-1) = p\cdot d\cdot  4^{k}(4^{k}+1)(2^{k}-1),\]
where $d,k\in\bN$ with $k\geq3$ odd. In both cases, using the same argument as above, we obtain a contradiction. This completes the proof.
\end{proof}

 
\subsection{Proof of Theorem~\ref{main thm2}} Let $n_0$ be as in the statement of the theorem. By Lemmas~\ref{lem1},~\ref{lem1'},~\ref{lem2}, and~\ref{lem2'}, conditions 1, 2, 3, 4 in Theorem~$\ref{general thm}$ are satisfied. Also, up to isomorphism there is only one insolvable group of order $n_0$  and it is non-abelian simple. It then follows from Proposition~\ref{prop2} (a) that Conjecture~\ref{conjecture} holds when $n=n_0$. We now deduce directly from Theorem~$\ref{general thm}$ that Conjecture~\ref{conjecture} also holds when $n = 2^r\cdot n_0$ for any $r\in\bN$.

\section{Algorithm to test the conjecture}\label{sec algorithm}

In this section, we shall describe an algorithm which may be used to prove Conjecture~\ref{conjecture} for a given $n$, as long as all finite groups of order $n$ are known. Then, we shall apply our algorithm to prove Theorem~\ref{main thm3}.

\vspace{1.5mm}

Recall that given any finite group $\Gamma$, the \emph{Fitting subgroup} of $\Gamma$, denoted by $\Fit(\Gamma)$, is the unique largest normal nilpotent subgroup of $\Gamma$. Plainly $\Fit(\Gamma)$ is a characteristic subgroup of $\Gamma$.

\begin{prop}\label{prop GN}Let $G$ and $N$ be two finite groups of the same order such that the set $\E'(G,N)$ is non-empty. Define
\begin{align*}\mathcal{M}(N) &=  \{|M| : M\mbox{ is a characteristic subgroup of }N\},\\
\mathcal{H}(G) &= \{|H| : H\mbox{ is a subgroup of }G\}.
\end{align*}
Then, we have $\mathcal{M}(N)\subset\mathcal{H}(G)$. Also, there is a solvable subgroup of $G$ whose order is that of $\Fit(N)$.
\end{prop}
\begin{proof}This follows directly from Propositions~\ref{prop1} (a) and~\ref{test2}.
\end{proof}

While Proposition~\ref{prop GN} gives us a way to test whether a pair $(G,N)$ satisfies condition (\ref{GN conditions}), applying it directly to prove Conjecture~\ref{conjecture} has two issues:
\begin{itemize}
\item Often there are many groups of a given order $n$, and it is inefficient to test whether (\ref{GN conditions}) holds for each pair $(G,N)$ of groups of order $n$. 
\item It is time-consuming to compute characteristic subgroups.
\end{itemize}
To overcome these difficulties, our idea is to let $G$ vary, and check that 
\begin{equation}\label{G conditions}
\E'(G,N)\neq\emptyset\mbox{ for some insolvable group $G$ of order $|N|$}
\end{equation}
cannot hold for each fixed $N$ separately. Also, we shall apply the test involving the Fitting subgroup first because it is the least time-consuming.

\vspace{1.5mm}

For $n\in\bN$, define the following sets:
\begin{align*}
\mathcal{L}_1(n) & = \bigcup_{\tiny\substack{|G|=n\\G\mbox{ is insolvable}}}\{|H|:H\mbox{ is a solvable subgroup of $G$}\},\\
\mathcal{L}_2(n) & =  \bigcup_{\tiny\substack{|G|=n\\G\mbox{ is insolvable}}}\{|H|:H\mbox{ is a subgroup of $G$}\}.
\end{align*}
Write $\mathcal{N}_0(n)$ for the set of all solvable groups of order $n$. For all $N\in\mathcal{N}_0(n)$:
\begin{itemize}
\item If $|\Fit(N)|\notin\mathcal{L}_1(n)$, then (\ref{G conditions}) does not hold by Proposition~\ref{prop GN}.
\item If $\Aut(N)$ is solvable, then $\Hol(N)$ is solvable by Lemma~\ref{solvable lem} and so it has no insolvable subgroup by Lemma~\ref{lem properties}, whence (\ref{G conditions}) does not hold.
\item If $n/2\in\mathcal{M}(N)$ and Conjecture~\ref{conjecture} holds for $n/2$, then (\ref{G conditions})  does not hold by Proposition~\ref{test2}, because any subgroup of index two (when it exists) of an insolvable group must be insolvable by Lemma~\ref{solvable lem}.
\item If $\mathcal{M}(N)\not\subset \mathcal{L}_2(n)$, then (\ref{G conditions}) does not hold by Proposition~\ref{prop GN}.
\item If the greatest common divisor of $n$ and $|\Out(N)|$ is solvable, then (\ref{G conditions}) does not hold by Proposition~\ref{test1}.
\end{itemize}
Our algorithm uses thee above criteria, and removes the groups $N\in\mathcal{N}_0(n)$ for which (\ref{G conditions}) fails to hold; if the set becomes empty, then Conjecture~\ref{conjecture} holds for $n$. More specifically, define the following sets:
\begin{align*}
\mathcal{N}_1(n) & = \{N\in\mathcal{N}_0(n): |\Fit(N)|\in\mathcal{L}_1(n)\},\\
\mathcal{N}_2(n) & = \{N\in\mathcal{N}_1(n): \Aut(N)\mbox{ is insolvable}\},\\
\mathcal{N}_{31}(n) & = \{N\in\mathcal{N}_2(n): n/2\not\in\mathcal{M}(N)\},\\
\mathcal{N}_{32}(n) & = \{N\in\mathcal{N}_2(n) : \mathcal{M}(N)\subset\mathcal{L}_2(n)\},\\
\mathcal{N}_{33}(n) & = \{N\in\mathcal{N}_2(n): \gcd(n,|\Out(N)|)\mbox{ is non-solvable}\}.
\end{align*}
If $\mathcal{N}_{32}(n)\cap\mathcal{N}_{33}(n)=\emptyset$, then Conjecture~\ref{conjecture} holds for $n$. Similarly, if Conjecture~\ref{conjecture} holds for $n/2$ and $\mathcal{N}_{31}(n)\cap\mathcal{N}_{32}(n)\cap\mathcal{N}_{33}(n)=\emptyset$, then Conjecture~\ref{conjecture} holds for $n/2$.

\vspace{1.5mm}

We have implemented the computations of the above sets, except $\mathcal{N}_{33}(n)$, in {\sc Magma} \cite{magma} and {\sf GAP} \cite{gap}. The code may be found in the appendix.


\subsection{Proof of Theorem~\ref{main thm3}}


The groups of order $n\leq 2000$ are available in the {\sc SmallGroups} Library \cite{SGlibrary}. Using this library, we ran our algorithm in {\sc Magma} to the non-solvable numbers $n\leq 2000$.

First, we computed that $\mathcal{N}_2(n)$ is empty except for
\[n = 480,600,960,1008,1200,1320,1344,1440,1512,1680,1800,1920.\]
Among these numbers, we further computed that $\mathcal{N}_{31}(n)\cap\mathcal{N}_{32}(n)$ is empty except for $n=1008,1512$. In fact, we have
\begin{align*}
\mathcal{N}_2(1008) & = \mathcal{N}_{31}(1008)\cap\mathcal{N}_{32}(1008) = \{\mbox{{\sc SmallGroup}}(1008,910)\},\\
\mathcal{N}_2(1512) & = \mathcal{N}_{31}(1512)\cap\mathcal{N}_{32}(1512) = \{\mbox{{\sc SmallGroup}}(1512,841)\}.
\end{align*}
Then, using the {\sc Magma} command {\tt OuterOrder}, we checked that $\mathcal{N}_{33}(1008)$ and $\mathcal{N}_{33}(1512)$ are empty. Thus, we now conclude that Conjecture~\ref{conjecture} indeed holds when $n\leq 2000$.

\vspace{1.5mm}

The calculations of $\mathcal{N}_2(n),\mathcal{N}_{31}(n)\cap\mathcal{N}_{32}(n)$ took a total of 22 min for all non-solvable numbers $n\leq 2000$. By contrast, it took a total of 231 min to confirm Conjecture~\ref{conjecture} directly by using the {\sc Magma} command {\tt RegularSubgroups} for all non-solvable numbers $n\leq 1000$ with $n\neq 480,672,960$. The calculations were done on an Intel Xeon CPU E5-1620 vs3 @ 3.5GHz machine with 16GB of RAM under Ubuntu 16.04LTS.

\vspace{1.5mm}

The cases $n=60,120,240,480,960,1920$ also follow from Theorem~\ref{main thm2}.

\section{Acknowledgments}

Part of this research was done while the authors visited each other at the University of Waikato and the  Yau Mathematical Sciences Center at Tsinghua University in year 2018. The visits were supported by the China Postdoctoral Science Foundation Special Financial Grant (no.: 2017T100060). We would like to thank both institutions. The first author would like to specially thank Prof. Daniel Delbourgo for his hospitality. 

\vspace{1.5mm}

The first author would like to thank Prof. Leandro Vendramin for pointing out that Theorem~\ref{thm1} (c) solves Problem 19.90 (d) in \cite{K book}.

\vspace{1.5mm}

Finally, we would like to thank the editor Prof. Eamonn O'Brien and the two referees for some very helpful suggestions. We would particularly like to thank one of the referees for pointing out a small gap in Theorem~$\ref{general thm}$ in the original manuscript.

\addresseshere

\begin{landscape}

\section*{Appendix: Computational codes}

{\sc Magma} code to compute $\mathcal{N}_1(n),\mathcal{N}_2(n),\mathcal{N}_{31}(n),\mathcal{N}_{32}(n)$:
\begin{lstlisting}[
  mathescape,
  columns=fullflexible,
  basicstyle=\ttfamily,
]
TestOrders:=[*any list of non-solvable numbers n which we wish to test*];
for n in TestOrders do
//Compute LL1 and LL2.
GG:=SmallGroups(n,func<x|not IsSolvable(x)>);
L1:=[];
L2:=[];
for G in GG do
Sub:=Subgroups(G);
  for H in Sub do
  order:=H`order;
    if IsSolvable(H`subgroup) then
    Append(~L1,order);
    end if;
    Append(~L2,order);
  end for;
end for;
LL1:=Set(L1);
LL2:=Set(L2);
//Compute NN1, NN2, NN31, NN32.
NN0:=[i:i in [1..#SmallGroups(n:Warning:=false)]|IsSolvable(SmallGroup(n,i))];
NN1:=[];
NN2:=[];
NN31:=[];
NN32:=[];
for i in NN0 do
N:=SmallGroup(n,i);
Fit:=FittingSubgroup(N);
//Determine whether N is in NN1.
  if Order(Fit) in LL1 then
  Append(~NN1,i);
  end if;
if i in NN1 then
Aut:=AutomorphismGroup(N);
//Determine whether N is in NN2.
  if not IsSolvable(Aut) then
  Append(~NN2,i);
  end if;
if i in NN2 then
Out:=[a:a in Generators(Aut)|not IsInner(a)];
NorSub:=NormalSubgroups(N);
CharSub:=[x:x in NorSub|forall{a:a in Out|a(x`subgroup) eq x`subgroup}];
MM:={M`order:M in CharSub}; 
//Determine whether N is in NN31.
  if n/2 notin MM then
  Append(~NN31,i);
  end if;
//Determine whether N is in NN32.
  if MM subset LL2 then
  Append(~NN32,i);
  end if;
end if;
end if;
end for;
//If NN2 is empty, then Conjecture 1.5 holds for n.
//If NN31 $\cap$ NN32 is empty, then Conjecture 1.5 holds for n as long as it holds for n/2.
//If NN31 $\cap$ NN32 is non-empty, then further test is required.
if IsEmpty(NN2) then
printf "Conjecture 1.5 holds for %o\n",n;
else
  if IsEmpty(NN32) then
  printf "Conjecture 1.5 holds for %o\n",n;
  else
    Inter:=Set(NN31) meet Set(NN32);
    if IsEmpty(Inter) then
    printf "Conjecture 1.5 holds for %o if it holds for %o\n",n,n/2;
    else
    print n,Inter;
    end if;
  end if;
end if;
end for;
\end{lstlisting}


{\sf GAP} code to compute $\mathcal{N}_1(n),\mathcal{N}_2(n),\mathcal{N}_{31}(n),\mathcal{N}_{32}(n)$:
\begin{lstlisting}[
  mathescape,
  columns=fullflexible,
  basicstyle=\ttfamily,
]
TestOrders:=[*any list of non-solvable numbers n which we wish to test*];;
for n in TestOrders do
GG:=Filtered(AllSmallGroups(n),G->not IsSolvable(G));
#Compute LL1 and LL2.
L1:=[];
L2:=[];
for G in GG do
Sub:=List(ConjugacyClassesSubgroups(G),Representative);
  for H in Sub do
  order:=Order(H);
    if IsSolvable(H) then
    Add(L1,order);
    fi;
    Add(L2,order);
  od;
od;
LL1:=Set(L1);
LL2:=Set(L2);
#Compute NN1, NN2, NN31, NN32.
NN0:=Filtered([1..Size(AllSmallGroups(n))],i->IsSolvable(SmallGroup(n,i)));
NN1:=[];
NN2:=[];
NN31:=[];
NN32:=[];
for i in NN0 do
N:=SmallGroup(n,i);
Fit:=FittingSubgroup(N);
#Determine whether N is in NN1.
  if Order(Fit) in LL1 then
  Add(NN1,i);
  fi;
if i in NN1 then
Aut:=AutomorphismGroup(N);
#Determine whether N is in NN2.
  if not IsSolvable(Aut) then
  Add(NN2,i);
  fi;
if i in NN2 then
CharSub:=CharacteristicSubgroups(N);
MM:=Set(CharSub,Order);
#Determine whether N is in NN31.
  if not n/2 in MM then
  Add(NN31,i);
  fi;
#Determine whether N is in NN32.
  if IsSubset(LL2,MM) then
  Add(NN32,i);
  fi;
fi;
fi;
od;
#If NN2 is empty, then Conjecture 1.5 holds for n.
#If NN31 $\cap$ NN32 is empty, then Conjecture 1.5 holds for n as long as it holds for n/2.
#If NN31 $\cap$ NN32 is non-empty, then further test is required.
if IsEmpty(NN2) then
Print("Conjecture 1.5 holds for ",n,"\n");
else
  if IsEmpty(NN32) then
  Print ("Conjecture 1.5 holds for ",n,"\n");
  else
    Inter:= Intersection(NN31,NN32);
    if IsEmpty(Inter) then
    Print("Conjecture 1.5 holds for ",n," if it holds for ",n/2,"\n");
    else
    Print(n,Inter,"\n");
    fi;
  fi;
fi;
od;
\end{lstlisting}


{\sc Magma} code to test Conjecture~\ref{conjecture} directly:

\begin{verbatim}
TestOrders:=[*any list of non-solvable numbers n which we wish to test*];
for n in TestOrders do
NN0:=[i:i in [1..#SmallGroups(n:Warning:=false)]|IsSolvable(SmallGroup(n,i))];
NN00:=[];
for i in NN0 do
N:=SmallGroup(n,i);
Hol:=Holomorph(N);
RegSub:=RegularSubgroups(Hol);
InsolRegSub:=[R:R in RegSub| not IsSolvable(R`subgroup)];
  if not IsEmpty(InsolRegSub) then
  Append(~NN0,i);
  end if;
end for;
//NN00 is empty if and only if Conjecture 1.5 holds for n.
print n,NN00; 
end for;
\end{verbatim}

\end{landscape}

\end{document}